\documentclass[10pt,reqno]{amsart}

\usepackage{amsmath} \usepackage{amssymb, amscd,cancel, graphicx,soul}
\usepackage{sseq}
\usepackage[new]{old-arrows}
	\usepackage{pinlabel,mathtools}

\usepackage{hyperref}
\hypersetup{
    colorlinks=true, 
    linktoc=all,     
    linkcolor=black,  
}

\usepackage{enumerate}

\usepackage{cleveref}
\usepackage{appendix}
\usepackage{tikz-cd}

\usepackage{mathdots}
\newtheorem{thm}{Theorem}[section]

\newtheorem{conj}[thm]{Conjecture}

\newtheorem{remark}[thm]{Remark}

\newtheorem*{ques*}{Question}
\newtheorem{defin}[thm]{Definition}

\newtheorem*{example*}{Example}
\newtheorem*{porism*}{Porism}
\newtheorem*{scholium*}{Scholium}

\newtheorem*{thm*}{Theorem}
\newtheorem*{defin*}{Definition}
\newtheorem*{lem*}{Lemma}
\newtheorem*{prop*}{Proposition}
\newtheorem*{remark*}{Remark}

\def\cC{{\mathcal C}}

\def\Kh{{\rm Kh}}

\def\Lee{{\rm Lee}}
\def\KM{{\rm KM}}

\def\bQ{{\mathbb Q}}

\def\bZ{{\mathbb Z}}

\newcommand{\ol}{\overline}

\begin{document}
\thispagestyle{empty}
\title[Khovanov concordance minima and the (4,5) torus knot.]{Khovanov concordance minima and the (4,5) torus knot.}
\author{Andrew Lobb} 
\address{Durham University,
UK.}
\email{andrew.lobb@durham.ac.uk}
\thanks{We thank Hayato Imori for explaining some of his work to us, and thank KAIST and OIST for their hospitality.}

\begin{abstract}
Ribbon concordance gives a partial order on knot types, and applying a knot homology functor to a ribbon concordance gives an inclusion of the homologies.  The question of the existence of global ribbon minima in each concordance class is a generalization of the slice-ribbon conjecture, which asserts that the unknot is the global minimum in its class.  We show that the (reduced rational) Khovanov homology of the (4,5) torus knot is a summand in the Khovanov homology of any knot in its concordance class.
\end{abstract}

\maketitle

\begin{center}
\includegraphics[width=0.2\linewidth]{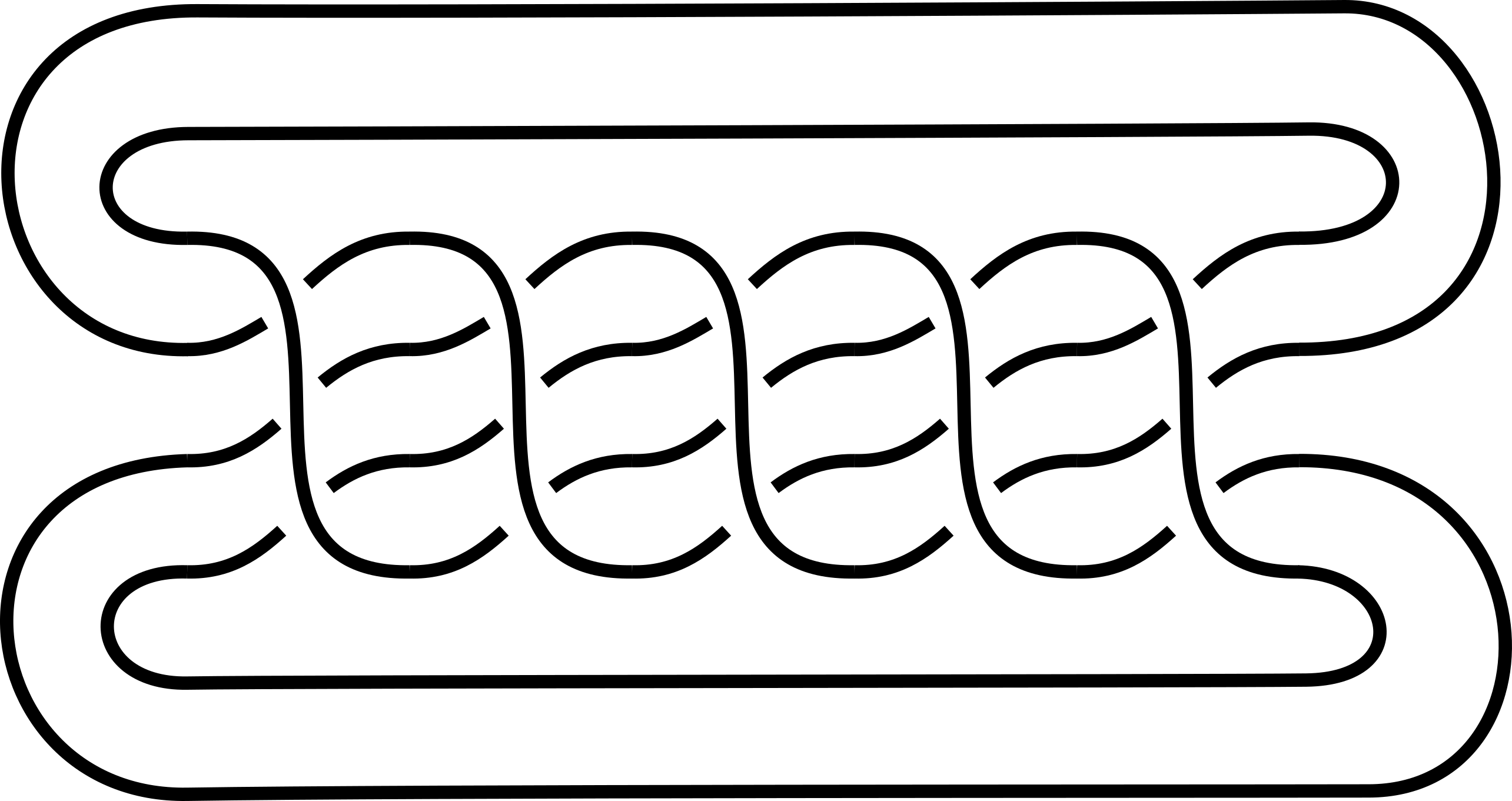}
\end{center}

\section{Introduction}
\label{sec:intro}
For oriented knots $K_0, K_1 \subset S^3$, we write $K_0 \leq K_1$ if there exists a concordance $\cC \colon K_0 \rightarrow K_1$ (that is, a smooth submanifold $S^1 \times [0,1] \cong \cC \subset S^3 \times [0,1]$ with $\partial \cC = K_0 \times \{0\} \sqcup K_1 \times \{ 1 \}$) such that the projection to $[0,1]$ is Morse with no local maxima.  We read this relation as $K_1$ has a \emph{ribbon concordance} to $K_0$. 

Agol \cite{agol} established a conjecture of Gordon's \cite{gordanribbon}: the relation $\leq$ satisfies
\[ K \leq K' \,\, {\rm and} \,\, K' \leq K \iff K = K' {\rm .} \]
Earlier, it was observed by Zemke \cite{zemke} that Heegaard Floer knot homology could obstruct ribbon concordance.  His argument generalizes in a straightforward manner to many other knot homologies (and maybe to every knot homology, suitably characterized) \cite{levinezemke,kang,daemilidmanvelavicksheawong}.  The observation is that if $\cC \colon K \rightarrow K'$ is a ribbon concordance, then one can consider the reverse concordance $\overline{\cC} \colon K' \rightarrow K$ (which may be no longer ribbon since every local minimum of $\cC$ becomes a local maximum of $\overline{\cC}$).  The composition $\ol{\cC} \circ \cC \colon K \rightarrow K$ gives a self-concordance of $K$.  The properties of many knot homologies, together with the topology of the composition $\ol{\cC} \circ \cC$, allow one to determine that the induced map
\[ H_*(\ol{\cC}) \circ H_*(\cC) = H_*(\ol{\cC} \circ \cC) \colon H_*(K) \longrightarrow H_*(K) \]
is an isomorphism (here we have written $H_*$ to stand for the reader's favorite knot homology).  Hence we conclude that $H_*(\cC)$ must include $H_*(K)$ in $H_*(K')$.  Consequently if $H_*(K)$ does not appear as a summand of $H_*(K')$ then we cannot have $K \leq K'$.

This motivates a definition.

\begin{defin}
	\label{defin} We write $K_0 \leq_{H_*} K_1$ if every concordance $\cC \colon K_0 \rightarrow K_1$ induces an inclusion as a summand $H_*(\cC) \colon H_*(K_0) \hookrightarrow H_*(K_1)$.
\end{defin}
Note that this definition depends on a choice of knot homology $H_*$.

\newpage

\section{Three conjectures and a theorem.}
\label{sec:conj}
The outstanding conjecture in ribbon concordance is the slice-ribbon conjecture:
\begin{conj}[Slice-ribbon conjecture.]
	\label{conj:sliceribbon}
	The concordance class of the unknot has a global minimum (the unknot itself) with respect to $\leq$.
\end{conj}

\noindent Extending this to other concordance classes gives a stronger conjecture:
\begin{conj}
\label{conj:globalsliceribbon}
Each concordance class has a global minimum with respect to $\leq$.
\end{conj}
\noindent And this suggests an algebraic counterpart:
\begin{conj}
	\label{conj:algsliceribbon}
	Each concordance class admits a global minimum with respect to $\leq_{H_*}$.
\end{conj}
Note that neither Conjecture \ref{conj:globalsliceribbon} nor \ref{conj:algsliceribbon} implies the other, but both imply the existence of a knot $K$ within each concordance class whose homology $H_*(K)$ appears as a summand of the homology of each knot in that class.  Broadly speaking, one could take the verification of cases of Conjecture \ref{conj:algsliceribbon} as evidence for the truth of Conjecture \ref{conj:globalsliceribbon}.  Global minima of $\leq_{H_*}$ are candidates for the unique minima of $\leq$ in their concordance class.

\begin{remark}
	\label{rem}
Suppose that $K$ is a global minimum for $\leq_{H_*}$ in its concordance class.  Then (assuming for example that $H_*$ takes values in finitely generated abelian groups or finite dimensional vector spaces) we have that any self-concordance $K \rightarrow K$ induces an isomorphism on $H_*(K)$.

Conversely, if any self-concordance $K \rightarrow K$ of a knot $K$ always induces an isomorphism on $H_*(K)$, then if $\cC \colon K \rightarrow K'$ is a concordance we see that $H_*(\ol{\cC}) \circ H_*(\cC) = H_*(\ol{\cC} \circ \cC)$ is an isomorphism.  Hence we can conclude that $K$ is a global minimum for $\leq_{H_*}$ in its concordance class.
\end{remark}

In this paper, we are particularly interested in the concordance classes of torus knots.  On the topological side, Gordon \cite{gordanribbon} proved that torus knots are local minima of $\leq$, and related results have been established recently, for example those in \cite{boningergreene,boningerribbon}.  On the algebraic side, Daemi-Scaduto \cite{daemiscaduto} proved that any self-concordance $\cC \colon K \rightarrow K$ of a torus knot $K$ induces an isomorphism on instanton knot homology (in many of its different flavors).  The class of L-space knots contains all torus knots, and Zemke \cite{zemkeb} showed that self-concordances of L-space knots also induce isomorphisms on Heegaard Floer knot homology.

Recently, Imori-Sano-Sato-Kaniguchi \cite{imori1,imori2} have shown that any 2-strand torus knot $T(2,n)$ is a global minimum in its concordance class with respect to $\leq_{\Kh_{h,t}}$ where $\Kh_{h,t}$ is any generalized Khovanov homology with coefficients in a polynomial ring over the integers.

In this paper we now turn our attention to one of the most vanilla flavors of knot homology: reduced Khovanov homology over the rationals $\bQ$, which we denote $\Kh$.  This is a finite dimensional rational vector space with a bigrading $(i,j) \in \bZ \oplus \bZ$: the homological grading $i \in \bZ$ and the quantum grading $j \in \bZ$.  If $U$ is the unknot then we take the normalization so that $\Kh(U)$ is $1$-dimensional and supported in bigrading $(0,0)$.

\begin{thm}
	\label{thm:t45}
	The $(4,5)$ torus knot $T$ is a global minimum for $\leq_{\Kh}$ in its concordance class.
\end{thm}

This theorem is the result of playing spectral sequences starting at $\Kh(T)$ off against each other, together with using cobordisms between $T$ and simple knots to constrain the differentials of these spectral sequences (the work finding cobordisms was carried out in \cite{LobbZentss}).  We believe that these techniques should be sufficient to establish Conjecture \ref{conj:algsliceribbon} for other classes of knots and for various other less vanilla flavors of Khovanov homology.

\section{Two spectral sequences and the (4,5) torus knot.}
\label{sec:ss}

Khovanov homology $\Kh$ appears as an early page of many spectral sequences, extensively studied in the literature.  One early example is the reduced version of the Lee spectral sequence $E_*^{\Lee}$ in which $E_2^{\Lee}(K) = \Kh(K)$ and $E_\infty^{\Lee}(K)$ is $1$-dimensional and supported in bigrading $(0,s(K))$ where $s(K) \in 2\bZ$ is the Rasmussen invariant \cite{lee,ras3}.

A more recent example (or, more precisely, many examples for different choices of filtration) is given by Kronheimer-Mrowka $E^\KM_*$ \cite{KrMrunknot,KrMrfilt}.  For this spectral sequence, $E_2^\KM(K) = \Kh(K)$ (we allow ourselves to shift the ordering of the pages so that the second page is the Khovanov page irrespective of the filtration chosen) and $E_\infty^\KM(K)$ is the reduced instanton knot homology $I^\natural(K)$ (or more precisely the associated graded vector space to the filtration on $I^\natural(K)$).

We now discuss some of what is known about these two spectral sequences.  The diagrams below each represent $\Kh(T)$, which has total dimension $9$.  A bullet represents a copy of $\bQ$.  The horizontal coordinate gives the homological grading $i$, while the $y$ coordinate gives the coordinate $j-i$ where $j$ is the quantum grading.  

\begin{minipage}{0.34\textwidth}
\begin{sseq}[ylabelstep=1,ylabels={12;;14;;16;}]{0...9}{0...5}
	\ssdropbull
	\ssmove 2 2
	\ssdropbull \ssname{a}
	\ssmove 1 1
	\ssdropbull
	\ssmove 2 2
	\ssdropbull
	\ssmove 2 0
	\ssdropbull
	\ssmove 2 0
	\ssdropbull \ssname{c}
	\ssmove {-1} {-1}
	\ssdropbull
	\ssmove {-2} {-2}
	\ssdropbull
	\ssmove {-2} 0
	\ssdropbull \ssname{b}
	\ssgoto a \ssgoto c \ssstroke \ssarrowhead
	\ssgoto b \ssgoto c \ssstroke \ssarrowhead
\end{sseq}
\end{minipage}
\begin{minipage}{0.6\textwidth}
The spectral sequence $E_*^\KM(T)$ has exactly one non-trivial differential after the Khovanov page $E_2^\KM(T) = \Kh(T)$, and this differential is of rank $1$ (see Section 11 of \cite{KrMrfilt}).  We indicate the two possible differentials (see Subsection 4.1 of \cite{LobbZentss}).
\end{minipage}

\begin{minipage}{0.34\textwidth}
	\begin{sseq}[ylabelstep=1,ylabels={12;;14;;16;}]{0...9}{0...5}
		\ssdropbull
		\ssmove 2 2
		\ssdropbull \ssname{a}
		\ssmove 1 1
		\ssdropbull \ssname{e}
		\ssmove 2 2
		\ssdropbull \ssname{f}
		\ssmove 2 0
		\ssdropbull \ssname{g}
		\ssmove 2 0
		\ssdropbull \ssname{h}
		\ssmove {-1} {-1}
		\ssdropbull \ssname{d}
		\ssmove {-2} {-2}
		\ssdropbull \ssname{c}
		\ssmove {-2} 0
		\ssdropbull \ssname{b}
		\ssgoto a \ssgoto e \ssstroke \ssarrowhead
		\ssgoto b \ssgoto f \ssstroke \ssarrowhead
		\ssgoto c \ssgoto g \ssstroke \ssarrowhead
		\ssgoto d \ssgoto h \ssstroke \ssarrowhead
	\end{sseq}
\end{minipage}
\begin{minipage}{0.6\textwidth}
	The spectral sequence $E_*^\Lee(T)$ converges to $E_\infty^\Lee(T)$, which is $1$-dimensional and supported in bigrading $(0,12)$.  The differentials after $E_2^\Lee(T) = \Kh(T)$ are determined by the fact that each has homological grading $1$.
\end{minipage}

If $\Sigma \colon K_0 \rightarrow K_1$ is an oriented knot cobordism with a smooth choice of basepoint for each link slice, then presenting $\Sigma$ as a movie $M$ between diagrams $D_0$ and $D_1$ we obtain spectral sequence morphisms:
\[ E_*^\KM(M) \colon E_*^\KM(D_0) \longrightarrow E_*^\KM(D_1) \,\,\, {\rm and} \,\,\, E_*^\Lee(M) \colon E_*^\Lee(D_0) \longrightarrow E_*^\Lee(D_1) {\rm .} \]
For the Lee spectral sequence, the map on $E_2^\Lee = \Kh$ agrees with the standard map induced by $M$ on Khovanov homology.  For the Kronheimer-Mrowka spectral sequence the map on $E_2^\KM = \Kh$ is only known to agree with the standard Khovanov map for movies not involving Reidemeister moves (although see \cite{imori1,imori2} for work in the non-reduced case).

We now let $\cC \colon T \rightarrow T$ be any self-concordance of $T$.  By Remark \ref{rem}, to establish Theorem \ref{thm:t45} we wish to see that $\cC$ induces an isomorphism on $\Kh(T)$.
In order to sidestep the Reidemeister question between $E_2^\KM$ and $\Kh$, we observe that the concordance $\cC$ can be represented as a basepointed movie between two diagrams $D$ and $D'$ of $T$, with moves in the following order:

\begin{itemize}
	\item $k$ $0$-handle attachments,
	\item Reidemeister moves,
	\item $(k+l)$ $1$-handle attachments,
	\item Reidemeister moves,
	\item $l$ $2$-handle attachments.
\end{itemize}

\newpage

\section{Khovanov homology and the (4,5) torus knot.}
\label{sec:45}

\begin{proof}[Proof of Theorem \ref{thm:t45}.]
The Kronheimer-Mrowka construction associates to this movie of $\cC \colon T \rightarrow T$ a morphism of spectral sequences, which decomposes as the composition of five morphisms, one for each of the bullet points above.  We write the decomposition of the movie as
\[ D = D_1 \longrightarrow D_2 \longrightarrow D_3 \longrightarrow D_4 \longrightarrow D_5 \longrightarrow D_6 = D' {\rm .} \]

On the Khovanov page $E_2^\KM$, the first map gives the inclusion
\[ \Kh(T) \cong \Kh(D=D_1) \longhookrightarrow \Kh(D_1 \sqcup U_k = D_2) = \Kh(D) \otimes \langle v_+, v_- \rangle^{\times k} \colon v \longmapsto v \otimes v_+^{\otimes k} \]
where we write $U_k$ for the $k$-component unlink.

The second map gives an isomorphism of bigrading $(0,0)$
\[ \Kh(D_2)\longrightarrow \Kh(D_3) {\rm .} \]
It is not known if this map coincides with the map as it is usually defined on Khovanov homology.  However, in each bigrading supporting the image of the first map, this is just an isomorphism of a $1$-dimensional vector space, since $Kh(T)$ has dimension at most $1$ in each homological grading.

The third map gives a linear map of bigrading $(0,-k-l)$
\[ \Kh(T) \otimes \langle v_+, v_- \rangle^{\otimes k} \cong \Kh(D_3) \longrightarrow Kh(D_4) \cong \Kh(T) \otimes \langle v_+, v_- \rangle^{\otimes l} {\rm ,} \]
which does coincide with the Khovanov map.  Note that this maps the image of the first map into the summand $\Kh(T) \otimes \langle v_- \rangle^{\otimes l}$.

The fourth map gives an isomorphism of bigrading $(0,0)$
\[ \Kh(D_4) \longrightarrow \Kh(D_5) {\rm ,} \]
and it is not known if this agrees with the usual Khovanov map.  We note again that, in each bigrading supporting $\Kh(T) \otimes \langle v_- \rangle^{\otimes l}$, this is an isomorphism of a $1$-dimensional vector space.

Finally, the fifth map is the projection
\[ \Kh(D_5 = D_6 \sqcup U_l) = \Kh(D_6) \otimes \langle v_+, v_- \rangle^{\otimes l} \longrightarrow \Kh(D_6 = D') \colon v \otimes v_-^{\otimes l} \longmapsto v {\rm .} \]

So we have observed that the map
\[ E_2^\KM(\cC) \colon \Kh(D) = E_2^\KM(D) \longrightarrow E_2^\KM(D') = \Kh(D')\]
in each non-zero bigrading corresponds to a composition of the Khovanov maps of $1$-dimensional vector spaces, together with isomorphisms of $1$-dimensional vector spaces.  It follows that the rank of this map agrees with the rank of the Khovanov map in each bigrading.

Since $E_\infty^\KM(\cC)$ is the map induced by $\cC$ on the (associated graded) instanton homology $I^\natural(T)$, and we know that this map is an isomorphism, it follows that $\Kh(\cC)$ is an isomorphism in all bigradings apart from possibly $(2,14)$, $(4,14)$, and $(9,17)$.

On the other hand, the generators in these gradings are each either the source or the target of a non-trivial differential in the spectral sequence $E_*^\Lee(T)$.  The other end of each differential is preserved under the Khovanov map.  Since $E_*^\Lee(\cC)$ gives a morphism of spectral sequences which is the Khovanov map on $E_2^\Lee(T) = \Kh(T)$, we conclude that $\Kh(\cC)$ is also an isomorphism in each bigrading $(2,14)$, $(4,14)$, and $(9,17)$.
\end{proof}

\bibliographystyle{amsplain}
\bibliography{References./works-cited.bib}

\providecommand{\bysame}{\leavevmode\hbox to3em{\hrulefill}\thinspace}
\providecommand{\MR}{\relax\ifhmode\unskip\space\fi MR }
\providecommand{\MRhref}[2]{%
  \href{http://www.ams.org/mathscinet-getitem?mr=#1}{#2}
}
\providecommand{\href}[2]{#2}
\begin{thebibliography}{10}

\bibitem{agol}
Ian Agol, \emph{Ribbon concordance of knots is a partial ordering}, Commun. Am.
  Math. Soc. \textbf{2} (2022), 374--379. \MR{4520779}

\bibitem{boningerribbon}
Joe Boninger, \emph{Positive knots and ribbon concordance}, Pacific J. Math.
  \textbf{335} (2025), no.~1, 81--95. \MR{4891413}

\bibitem{boningergreene}
Joe Boninger and Joshua~Evan Greene, \emph{Special alternating knots are band
  prime}, Int. Math. Res. Not. IMRN (2024), no.~10, 8758--8763. \MR{4749185}

\bibitem{daemilidmanvelavicksheawong}
Aliakbar Daemi, Tye Lidman, David~Shea Vela-Vick, and C.-M.~Michael Wong,
  \emph{Ribbon homology cobordisms}, Adv. Math. \textbf{408} (2022), Paper No.
  108580, 68. \MR{4467148}

\bibitem{daemiscaduto}
Aliakbar Daemi and Christopher Scaduto, \emph{Chern-{S}imons functional,
  singular instantons, and the four-dimensional clasp number}, J. Eur. Math.
  Soc. (JEMS) \textbf{26} (2024), no.~6, 2127--2190. \MR{4742808}

\bibitem{gordanribbon}
C.~McA. Gordon, \emph{Ribbon concordance of knots in the {$3$}-sphere}, Math.
  Ann. \textbf{257} (1981), no.~2, 157--170. \MR{634459}

\bibitem{imori1}
Hayato Imori, Taketo Sano, Kouki Sato, and Masaki Taniguchi, \emph{Cobordism
  maps in singular instanton homology and {K}hovanov homology {I}},
  journal={{\tt arXiv:2505.12095v1}}.

\bibitem{imori2}
\bysame, \emph{Cobordism maps in singular instanton homology and {K}hovanov
  homology {I}{I}}, journal={{\tt arXiv:2510.09399v2}}.

\bibitem{kang}
Sungkyung Kang, \emph{Link homology theories and ribbon concordances}, Quantum
  Topol. \textbf{13} (2022), no.~1, 183--205. \MR{4404800}

\bibitem{KrMrunknot}
P.~B. Kronheimer and T.~S. Mrowka, \emph{Khovanov homology is an
  unknot-detector}, Publ. Math. Inst. Hautes \'Etudes Sci. (2011), no.~113,
  97--208. \MR{2805599}

\bibitem{KrMrfilt}
Peter Kronheimer and Tomasz Mrowka, \emph{Filtrations on instanton homology},
  Quantum Topol. \textbf{5} (2014), no.~1, 61--97. \MR{3176310}

\bibitem{lee}
E.S. Lee, \emph{{An endomorphism of the {K}hovanov invariant}}, Adv. Math.
  \textbf{197} (2002), no.~2, 554--586. \MR{2173845 (2006g:57024)}

\bibitem{levinezemke}
Adam~Simon Levine and Ian Zemke, \emph{Khovanov homology and ribbon
  concordances}, Bull. Lond. Math. Soc. \textbf{51} (2019), no.~6, 1099--1103.
  \MR{4041014}

\bibitem{LobbZentss}
Andrew Lobb and Raphael Zentner, \emph{On spectral sequences from {K}hovanov
  homology}, Algebr. Geom. Topol. \textbf{20} (2020), no.~2, 531--564.
  \MR{4092306}

\bibitem{ras3}
J.~Rasmussen, \emph{Khovanov homology and the slice genus}, Invent. Math.
  \textbf{182} (2010), 419--447.

\bibitem{zemke}
Ian Zemke, \emph{Knot {F}loer homology obstructs ribbon concordance}, Ann. of
  Math. (2) \textbf{190} (2019), no.~3, 931--947. \MR{4024565}

\bibitem{zemkeb}
\bysame, \emph{Link cobordisms and absolute gradings on link {F}loer homology},
  Quantum Topol. \textbf{10} (2019), no.~2, 207--323. \MR{3950650}

\end{thebibliography}
%
\end{document}